\newif\ifstandardtemplate \standardtemplatetrue
\newif\ifelseviertemplate \elseviertemplatefalse
\newif\ifspringertemplate \springertemplatefalse
\newif\ifwileytemplate    \wileytemplatefalse
\newcommand{\mytitle}{Connecting beams and continua: variational basis and mathematical analysis}
\newcommand{\myabstract}{We present a new variational principle for 
  linking models of beams and deformable solids, providing also its
  mathematical analysis.
  Despite the apparent differences between the two types of
  governing equations, it will be shown that
  the equilibrium of systems combining beams
  and solids can be obtained from a joint constrained variational principle
  and that the resulting boundary-value problem is well posed.}
\newcommand{\myack}{Partial support for this work has been provided
  by project DPI2017-92526-EXP from the Spanish Ministry of Science,
  Innovation, and Universities.}

\ifspringertemplate
\documentclass[smallcondensed]{svjour3}
\usepackage[latin1]{inputenc}
\usepackage[final]{changes}
\usepackage{amssymb}
\usepackage{amsmath}
\usepackage{graphicx}
\usepackage{enumitem}

\title{\mytitle
  \thanks{\myack}}
\author{Ignacio Romero}

\journalname{Journal of Elasticity}

\institute{I. Romero \at
  Universidad Polit\'ecnica de Madrid, Jos\'e Guti\'errez Abascal, 2, Madrid 29006, Spain\\
  \email{ignacio.romero@upm.es}
  \and
  IMDEA Materials Institute, Eric Kandel 2, Tecnogetafe, Madrid 28906, Spain\\
  \email{ignacio.romero@imdea.org}
  \and
  Orcid: 0000-0003-0364-6969}

\titlerunning{\mytitle}
\authorrunning{I. Romero}

\date{Received: date / Accepted: date}

\begin{document}
\maketitle

\begin{abstract}
  \myabstract
\end{abstract}

\smartqed

\fi

%
%
\ifelseviertemplate
\begin{frontmatter}

\title{\mytitle}

\author[1,2]{Ignacio Romero\corref{cor1}}
\ead{ignacio.romero@upm.es}

\address[1]{Universidad Polit\'ecnica de Madrid, Spain}
\address[2]{IMDEA Materials Institute, Spain}
\address[3]{xx}

\cortext[cor1]{Corresponding author. ETS Ingenieros Industriales,
  Jos\'e Guti\'errez Abascal, 2, Madrid 28006, Spain}
\runninghead{Romero and Olleros}

\begin{abstract}
\myabstract
\end{abstract}
\end{frontmatter}

\fi

%
%
\ifwileytemplate
\documentclass[doublespace,times]{nmeauth}
\usepackage[latin1]{inputenc}
\usepackage[final]{changes}
\usepackage[latin1]{inputenc}
\usepackage{amssymb}
\usepackage{amsmath}
\usepackage{amsthm}
\usepackage{graphicx}
\usepackage{enumitem}
\usepackage{natbib}
\usepackage{xfrac}
\usepackage{todonotes}
\usepackage{siunitx}
\graphicspath{{Figures/}{figures/}}
\message{Compiling with Wiley template}

\begin{document}
\runningheads{I. Romero}{...}

\title{\mytitle}
\author{Ignacio Romero\affil{1}\affil{2}\corrauth}

\address{\affilnum{1} ETSII, Universidad Politécnica de Madrid, 
         Jos\'{e} Guti\'{e}rrez Abascal, 2, 28006 Madrid, Spain\break
         \affilnum{2} IMDEA Materials Institute, Eric Kandel 2, 28096 Getafe, Madrid, Spain}

\corraddr{Dpto. de Ingenier\'{\i}a Mec\'{a}nica; E.T.S. Ingenieros Industriales; 
José Gutiérrez Abascal, 2; 28006 Madrid; Spain. Fax (+34) 91 336 3004}

\begin{abstract} 
\end{abstract}

\maketitle
\fi

%
%
\ifstandardtemplate
\documentclass[10pt,a4paper]{article}
\usepackage{amssymb}
\usepackage{amsmath}
\usepackage{amsthm}
\usepackage{graphicx}
\usepackage{enumitem}
\usepackage{xfrac}
\usepackage{todonotes}
\usepackage{siunitx}
\graphicspath{{Figures/}{figures/}}

\message{Compiling with default template}

\title{\mytitle}
\author{I. Romero$^1$}
\date{$^1$Universidad Polit\'ecnica de Madrid, 
  Jos\'{e} Guti\'{e}rrez Abascal, 2, 28006 Madrid, Spain\\[2ex]%
  $^2$ IMDEA Materials Institute, Eric Kandel,2; 28906 Madrid, Spain\\[2ex]
  $^3$Organization \\[2ex]%
    \today}

\newtheorem{theorem}    {Theorem}[section]
\newtheorem{lemma}      [theorem]{Lemma}
\newtheorem{corollary}  [theorem]{Corollary}
\newtheorem{proposition}[theorem]{Proposition}
\newtheorem{algorithm}  [theorem]{Algorithm}

\theoremstyle{definition}
\newtheorem{definition} [theorem]{Definition}
\newtheorem{problem}    [theorem]{Problem}

\newtheorem{examplex}[theorem]{$\triangleright\;$Ejemplo}
\newenvironment{example}{\medskip\begin{examplex}}{\hfill$\triangleleft$\end{examplex}}

\theoremstyle{remark}
\newtheorem{remark}{Remark}
\newtheorem{remarks}{Remarks}
\newtheorem{note}{Note}

\begin{document}
\maketitle
\fi

\newcommand{\mbs}[1]{\boldsymbol{#1}}
\newcommand{\concept}[1]{\textbf{\emph{#1}}}
\newcommand{\defined}{:=}
\newcommand{\dev}{{\mathop{\mathrm{dev}}}}
\renewcommand{\div}{{\mathop{\mathrm{div}}}}
\newcommand{\pairing}[2]{\langle{#1},{#2}\rangle}
\newcommand{\dd}[2]{\frac{\mathrm{d} #1}{\mathrm{d} #2}}
\newcommand{\pd}[2]{\frac{\partial #1}{\partial #2}}
\newcommand{\fd}[2]{\frac{\delta #1}{\delta #2}}
\newcommand{\set}[1]{\left\{#1\right\}}
\newcommand{\trace}{{\mathop{\mathrm{tr}}}}
\newcommand{\norm}[1]{\|#1\|}

\let\oldGamma=\Gamma \renewcommand{\Gamma}{\mathit{\oldGamma}}
\let\oldLambda=\Lambda \renewcommand{\Lambda}{\mathit{\oldLambda}}
\let\oldOmega=\Omega \renewcommand{\Omega}{\mathit{\oldOmega}}
\newcommand{\bigO}{\mathcal{O}}

\newcommand{\normbeam}[1]{\;\norm{#1}_{\mathcal{W}\times\mathcal{R}}}
\newcommand{\norma}[1]{\;|\kern-1pt\norm{#1}\kern-1pt|}
\newcommand{\normu}[1]{\;\norm{#1}_{\mathcal{U}}}
\newcommand{\normv}[1]{\;\norm{#1}_{\mathcal{V}}}
\newcommand{\normq}[1]{\;\norm{#1}_{\mathcal{Q}}}
\newcommand{\uptohere}{\centerline{\textcolor{blue}{\rule{6cm}{0.2cm}}}}

%
%

\section{Introduction}
\label{sec-intro}

The problems of beams and deformable solids refer both to the mechanical response of bodies when
subjected to the external actions, including forces, torques, and imposed displacements.  However,
from the mathematical viewpoint, these two problems are intrinsically very different. Even when
restricted to small strains, the kinematics of these two types of bodies are disparate: whereas the
former is described by a displacement field on an open set of two or three-dimensional Euclidean
space, the latter depends on the displacement and the rotation on an interval of the real line. The
equilibrium equations of a deformable solid, moreover, are \emph{partial} differential equations, in
contrast with the \emph{ordinary} differential equations that describe the equilibrium of forces and
momenta in a beam.

Despite the apparent differences between the mathematical description of the mechanics of beams and
deformable solids, there are deep relations between  them. After all, beams are nothing but a
special class of solids whose equations can be obtained from the equations of solid mechanics by
exploiting some asymptotic behavior or by constraining the class of admissible kinematics (see, for
example, \cite{Green:rodI:1974,Green:rodII:1974} for a description of these two avenues for model
reduction).

One specific aspect that is of both theoretical and practical interest is the combination of the
equations of beams and solids within a single mechanical system or structure. From the theoretical
point of view, the interest lies in the formulation of links between these two types of equations
and the well-posedness of the resulting boundary-value problems. From the practical side, joint
beam/solid equations lead to numerical methods that can efficiently represent the behavior of (beam)
structures with subsets studied as three-dimensional solids.

Recently, the author has presented novel formulations of coupled beam/solid mechanics that lead to
numerical methods, both in the linear and nonlinear regimes~\cite{Romero:2018iu}.  These
formulations, based on new variational principles, can be easily discretized using, for example,
finite elements, and replace commonly employed \emph{ad hoc} links between beams and solids (e.g.,
\cite{Surana:1980uy,Surana:1980uo,Gmur:1993tk}). The latter, often based on constraints on the
discrete solution, lack a variational basis and thus neither their well-posedness nor their
stability can be ascertained.

In this article we study  boundary-value problems of linked, deformable, beams and solids in the
context of linearized elasticity, as defined by a constrained variational principle.  The main goal
is to prove the well-posedness of problems with beams and solids involving the minimum set of
boundary conditions, effectively proving that the linking terms provide the right stability to the
equations, precluding rigid body motions of the system. The boundary-value problems that will be
studied have the structure of saddle-point optimization programs in Hilbert spaces (e.g.,
\cite{Boffi:2013jt}) and standard analysis techniques can be used to study their stability and
well-posedness.

In section~\ref{sec-statement} we summarize the equations that govern deformable solids and beams in
the context of small strain kinematics, highlighting the variational statement of these two problems
and their essential mathematical properties. Section~\ref{sec-joint} formulates the simplest problem
consisting of a beam and a solid that share an interface with the minimum set of Dirichlet boundary
conditions. A joint variational principle, where the kinematic compatibility is introduced with
Lagrange multipliers, is presented as well. The well-posedness of the resulting boundary-value
problem is analyzed in Section~\ref{sec-analysis}. The article concludes with a
summary of the main results in Section~\ref{sec-summary}.

\section{Problem statement}
\label{sec-statement}
This article analyses boundary value problems of joint
continuum solids and beams whose solutions correspond to the mechanical equilibrium
of both types of bodies, as well as certain compatibility relations in their shared
interfaces. Before formulating the global problem, the governing equations
of elasticity and beams are briefly reviewed, and their main mathematical
properties are identified.

The choice of boundary conditions in these problems is crucial. To show that
the constraints that are later introduced effectively link beams and solids, we
will present the pure traction problem of an elastic solid and a mixed traction-displacement
problem of a beam. Later, we will prove that these two bodies, when
appropriately connected, result in a stable structure.

\subsection{The Neumann problem of small strain, elastic solids}
\label{subs-solid}
We start by describing the continuum solid, and we restrict our presentation
to an elastic  one that occupies a bounded open set $\mathcal{B}\subset\mathbb{R}^3$
with volume $|\mathcal{B}|$.
The boundary of the solid is denoted $\partial\mathcal{B}$ and we identify a
subset $\Sigma\subsetneq\partial\mathcal{B}$ that will later be linked to a beam.

In classical elasticity, the unknown is the displacement  $\mbs{u}\in \mathcal{U} :=
\left[H^1(\mathcal{B})\right]^3$, the Hilbert space of vectors fields with (Lebesgue)
square-integrable components and square-integrable (weak) first derivatives. The stored energy of
the deformable body is given by a  scalar function $W=\hat{W}( \mbs{\varepsilon})$, where
$\mbs{\varepsilon}= \nabla^s \mbs{u} := \frac{1}{2}(\nabla \mbs{u} + \nabla^T \mbs{u})$ is the
infinitesimal strain tensor and $\nabla$ is the gradient operator.  More specifically, for linear
isotropic materials this function takes the form  $\hat{W}(\mbs{\varepsilon}) = \mu
\mbs{\varepsilon}\cdot \mbs{\varepsilon} +\frac{\lambda}{2} \trace[\mbs{\varepsilon}]^2$ where
$\lambda,\mu$ are the two Lam\'e constants, the dot product refers to the complete index
contraction, and $\trace[\cdot]$ is the trace operator.

Considering that the body might be subject to body forces $\mbs{f}\in
[H^{-1}(\mathcal{B})]^3$ and surface tractions $\mbs{t}$ on $\partial\mathcal{B}\setminus \Sigma$,
the total potential energy of the body is 
\begin{equation}
  \Pi_{\mathcal{B}} (\mbs{u})
  \defined
  \frac{1}{2} a_{\mathcal{B}}( \mbs{u}, \mbs{u}) - f_{\mathcal{B}}(\mbs{u})\ ,
   \label{eq-potential-body}
\end{equation}
with 
\begin{subequations}
  \begin{align}    
  a_{\mathcal{B}}( \mbs{u}, \mbs{v})
    &\defined
      \int_{\mathcal{B}}
      \left( 2\mu \nabla^s \mbs{u}\cdot\nabla^s\mbs{v} +
      \lambda (\nabla\cdot\mbs{u})\,( \nabla\cdot \mbs{v})
      \right)
      \,\mathrm{d} V\ ,
      \label{eq-solid-a}
    \\
    f_{\mathcal{B}}(\mbs{u})
  &\defined
    \int_{\mathcal{B}} \mbs{f}\cdot \mbs{u} \,\mathrm{d} V
    +
    \int_{\partial\mathcal{B}\setminus\Sigma} \mbs{t}\cdot \mbs{u} \,\mathrm{d} A ,
    \label{eq-solid-f}
  \end{align}
\end{subequations}
for all $\mbs{u},\mbs{v}\in \mathcal{U}$. We note, in passing, that the
potential energy~\eqref{eq-potential-body} might not have any minimiser in $\mathcal{U}$ 
--- if the forces are not equilibrated --- or alternatively, have an infinite
number of them, since the displacement function has no imposed values
at the boundary \cite{Gurtin:lin72}. If studied by itself, the minimization of the
potential~\eqref{eq-potential-body} corresponds to the Neumann problem
of elasticity and the right functional analysis setting corresponds to the
quotient space of $\mathcal{U}$ modulo the set of infinitesimal rigid
body motions.

To set up the  analysis framework for the study
of three-dimensional solids, we first recall the norm
on the space $\mathcal{U}$ which has the standard form
\begin{equation}
  \norm{\mbs{u}}_{\mathcal{U}} 
  \defined
  \left(
    \norm{\mbs{u}}^2_{[L^2(\Omega)]^3} 
    +
    L^2
    \norm{\nabla\mbs{u}}^2_{[L^2(\Omega)]^3} 
  \right)^{1/2}
  \ .
  \label{eq-unorm}
\end{equation}
The bilinear form \eqref{eq-solid-a} verifies the following continuity
and stability bounds
\begin{subequations}\label{solid-bounds}
  \begin{align}
    |a_{\mathcal{B}}(\mbs{u},\mbs{v})|
    &\le 
    C_{\mathcal{B}} \normu{\mbs{u}} \normu{\mbs{v}},
      \label{eq-solid-bounds-1}
    \\
    a_{\mathcal{B}}(\mbs{u},\mbs{u}) + \norm{\mbs{u}}^2_{[L^2(\mathcal{B})]^3}
    &\ge
    \alpha_{\mathcal{B}} \normu{\mbs{u}}^2\ ,
      \label{eq-solid-bounds-2}
  \end{align}
\end{subequations}
for some positive constants $C_{\mathcal{B}},\alpha_{\mathcal{B}}$, and all
$\mbs{u},\mbs{v}\in\mathcal{U}$.  It bears emphasis that, due to the lack of Dirichlet boundary
conditions on the boundary of the body, the bilinear form $a_{\mathcal{B}}(\cdot,\cdot)$ is not
coercive in $\mathcal{U}$. Rather,  and based on Korn's second inequality \cite{Marsden:1983ty},
only the weaker statement~\eqref{eq-solid-bounds-2} can be made. Also, the linear form
$f_{\mathcal{B}}$ is assumed to be continuous, i.e.,
\begin{equation}
    f_{\mathcal{B}}(\mbs{u}) 
    \le
    c_{\mathcal{B}} \normu{\mbs{u}}\ ,
   \label{eq-solid-fbound}
\end{equation}
with $c_{\mathcal{B}}>0$ for all $\mbs{u}\in\mathcal{U}$.

\subsection{Beam mechanics}
\label{subs-beam}

A cantilever beam of length $L$ is now studied. Its curve of centroids is described by a known
smooth curve $\mbs{r}:[0,L]\to\mathbb{R}^3$, with a cross section attached to each point of the
curve and oriented according to a known smooth rotation field $\mbs{\Lambda}:[0,L]\to SO(3)$, the
latter referring to the set of proper orthogonal tensors. The points on $\mbs{r}$ and sections
$\mbs{\Lambda}$ are parameterized by the arclength $s\in[0,L]$ and we choose $s=0$ and $s=L$ to
correspond, respectively, to the clamped section and free tip.

Let $\{\mbs{e}_1,\mbs{e}_2,\mbs{e}_3\}$ be a Cartesian basis. Then $\mbs{\Lambda}(s)$ maps
$\mbs{e}_3$ to the unit tangent vector to curve of centroids at the point $\mbs{r}(s)$, and
$\{\mbs{e}_1,\mbs{e}_2\}$ to the directions of the principal axis of the cross section at the same
point.  The displacement of the centroids will be given by the vector field $\mbs{w}\in \mathcal{W}
:= [H^1_0(0,L)]^3$ and the incremental rotation vector of the cross sections as
$\mbs{\theta}\in\mathcal{R} := [H^1_0(0,L)]^3$.  Following our previous notation, $[H_0^1(0,L)]^3$
refers to the Hilbert space of vectors fields on $(0,L)$ with vanishing trace at $s=0$.

Shear deformable, three-dimensional beams employ two deformation measures,
namely,
\begin{equation}
  \begin{aligned}
    \mbs{\Gamma} &= 
    \hat{\mbs{\Gamma}}(\mbs{w},\mbs{\theta})
    \defined
    \mbs{\Lambda}^T (\mbs{u}' + \mbs{\theta}\times \mbs{r}')\ ,
    \\
    \mbs{\Omega} &= 
    \widehat{\mbs{\Omega}}(\mbs{\theta})
    \defined
    \mbs{\Lambda}^T \mbs{\theta}'\ , \\ 
  \end{aligned}
   \label{eq-def-timoshenko}
\end{equation}
where the prime symbol denotes the derivative with respect to the arc-length.
The strain $\mbs{\Gamma}$ holds the shear and axial deformations,
whereas the vector  $\mbs{\Omega}$ contains the bending curvatures and the torsion deformation.

The simplest \emph{section} constitutive law for a beam of a linear elastic
and isotropic material with Young's and shear moduli~$E,G$, respectively,
is based on a quadratic stored energy function per unit length.
It has the form 
\begin{equation}
  U(\mbs{\Gamma},\mbs{\Omega})
  \defined
  \frac{1}{2} \mbs{\Gamma}\cdot \mbs{C}_{\mbs{\Gamma}} \mbs{\Gamma}
  +
  \frac{1}{2} \mbs{\Omega}\cdot \mbs{C}_{\mbs{\Omega}} \mbs{\Omega}
  \label{eq-stored}
\end{equation}
with section stiffness $\mbs{C}_\Gamma = \mathrm{diag}[GA_1,GA_2,EA]$ and
$\mbs{C}_\Omega = \mathrm{diag}[EI_1,EI_2,GI_t]$, where $A$ is the cross
section area, $A_1,A_2$ are the (shear) reduced sections areas in the
two principal directions, $I_1,I_2$ are the two principal moments
of inertia, and $I_t$ is the torsional inertia.
When the beam is under distributed loads and moments, denoted
respectively as $\bar{\mbs{n}}$ and $\bar{\mbs{m}}$,
and subject to a concentrated load $\bar{\mbs{P}}$ and
moment $\bar{\mbs{Q}}$ at the tip, its total  potential energy can be expressed as
\begin{equation}
  \Pi_b( \mbs{w}, \mbs{\theta})
  \defined
  \frac{1}{2} a_b(\mbs{w}, \mbs{\theta}; \mbs{w}, \mbs{\theta}) 
  -
  f_b(\mbs{w},\mbs{\theta})\ ,
   \label{eq-pot-beam}
\end{equation}
with $(\mbs{w},\mbs{\theta})\in\mathcal{W}\times\mathcal{R}$ and
\begin{equation}
  \begin{aligned}
    a_b(\mbs{w}, \mbs{\theta}; \mbs{t}, \mbs{\beta})
    &\defined
    \int_0^L 
    \left(
      \hat{\mbs{\Gamma}}(\mbs{w},\mbs{\theta})\cdot \mbs{C}_{\Gamma}
      \hat{\mbs{\Gamma}}(\mbs{t},\mbs{\beta}) 
      +
      \hat{\mbs{\Omega}}(\mbs{\theta})\cdot \mbs{C}_{\Omega}
      \hat{\mbs{\Omega}}(\mbs{\beta})
    \right)
    \,\mathrm{d} S ,
    \\
    f_b(\mbs{t},\mbs{\beta})
    &\defined
    \int_0^L 
    \left(
      \bar{\mbs{n}}\cdot \mbs{t} 
      +
      \bar{\mbs{m}}\cdot \mbs{\beta}
    \right)
    \,\mathrm{d} S
    + \bar{\mbs{P}}\cdot \mbs{u}_*
    +\bar{\mbs{Q}}\cdot \mbs{\theta}_*
    \ ,
  \end{aligned}
   \label{eq-forms-beam}
\end{equation}
for all $(\mbs{t},\mbs{\beta})\in \mathcal{W}\times\mathcal{R}$,
and $\mbs{u}_*:=\mbs{u}(L),\;\mbs{\theta}_*=\mbs{\theta}(L)$. 

To set up the functional setting for the beam problem, we recall
the norms on the space of displacements and rotations which are
\begin{equation}
  \begin{split}
    \norm{\mbs{w}}_{\mathcal{W}}
    &\defined
    \left(
      \norm{\mbs{w}}^2_{[L^2(0,L)]^3} 
      + L^2
      \norm{\mbs{w}'}^2_{[L^2(0,L)]^3} 
    \right)^{1/2}
    \ , \\
    \norm{\mbs{\theta}}_{\mathcal{R}} 
    &\defined
    \left(
      \norm{\mbs{\theta}}^2_{[L^2(0,L)]^3}
      +  L^2
      \norm{\mbs{\theta}'}^2_{[L^2(0,L)]^3}
    \right)^{1/2}
  \ .
   \end{split}
   \label{eq-norms}
\end{equation}
Also, the product space $\mathcal{W}\times\mathcal{R}$, the natural
setting for the beam problem, has the product norm
\begin{equation}
  \normbeam{(\mbs{w},\mbs{\theta})}
  =
  \left(
    \norm{\mbs{w}}^2_{\mathcal{W}}
    +
    L^2
    \norm{\mbs{\theta}}^2_{\mathcal{R}}
  \right)^{1/2}.
   \label{eq-product-norm}
\end{equation}
The bilinear form~\eqref{eq-forms-beam} verifies the 
continuity and stability bounds
\begin{equation}
  \begin{aligned}
    |a_b(\mbs{w},\mbs{\theta}; \mbs{t}, \mbs{\beta})|
    &\le
    C_{b} \normbeam{(\mbs{w},\mbs{\theta})} \normbeam{(\mbs{t},\mbs{\beta})}\ ,
    \\
    a_b(\mbs{w},\mbs{\theta}; \mbs{w}, \mbs{\theta})
    &\ge
    \alpha_b \normbeam{(\mbs{w},\mbs{\theta})}^2\ ,
  \end{aligned}
   \label{eq-beam-bounds}
\end{equation} 
for some constants $C_b,\alpha_b>0$ and all
$(\mbs{w},\mbs{\theta}),(\mbs{t},\mbs{\beta})\in\mathcal{W}\times\mathcal{R}$.  In contrast with the
bilinear form of the solid, and precisely due to the boundary conditions on the beam,
the bilinear form $a_b(\cdot,\cdot)$ is coercive in $\mathcal{W}\times\mathcal{R}$.
The linear form $f_b$ will be assumed to be continuous as well, i.e., there exists a
constant $c_b>0$ such that for all
$(\mbs{t},\mbs{\eta})\in\mathcal{W}\times\mathcal{R}$
\begin{equation}
  f_b(\mbs{t},\mbs{\eta})
  \le
  c_b
  \normbeam{(\mbs{t},\mbs{\eta})}\ .
  \label{eq-beam-fbound}
\end{equation}

\section{Joint formulation of  solids and beams}
\label{sec-joint}

We consider now the formulation of a problem in which a beam and a three-dimensional solid,
connected at some interface, deform to reach equilibrium under the action of external
forces. Two issues need to be discussed. First, the minimal \emph{compatibility conditions} that can
be used to link the kinematics of the beam and the solid on their shared interface.
Second, the \emph{stability}  and \emph{well-posedness} of the global problem under
the smallest set of Dirichlet boundary conditions.

The first issue will be addressed in this section, and follows our previous work
\cite{Romero:2018iu}. The second issue in studied in Section~\ref{sec-analysis}. To analyse both of
them, we consider the simplest case, an elastic solid as the one described in
Section~\ref{subs-solid}, devoid of Dirichlet boundary conditions, attached through a surface
$\Sigma$ to the tip of a cantilever beam, of the type defined in Section~\ref{subs-beam}.  The
number of Dirichlet boundary conditions for the \emph{global} problem is thus six, and it remains to
be proven that, when the right links are employed, the former suffice to ensure the stability of the
problem. Other, apparently more complex situations (with more beams or solids), are essentially
equivalent to this one.

\subsection{Link formulation}
\label{subs-link}
We define next two constraints relating the displacement and rotation vector of the beam at the
free end, denoted respectively as $\mbs{w}_*$ and $\mbs{\theta}_*$, with the displacement field
$\mbs{u}$ of the body on the connected surface~$\Sigma$.
More precisely, the first constraint imposes that
the tip displacement is equal to 
the average displacement of the body on $\Sigma$, that is
\begin{equation}
  \mbs{w}_* = \frac{1}{|\Sigma|} \int_\Sigma \mbs{u}\,\mathrm{d} A\ .
  \label{eq-constraint-w}
\end{equation}
The second constraint imposes that the rotation at the
tip of the beam, indicated as $\mbs{\theta}_*$, is identical
to the average surface rotation on $\Sigma$. To express it,
consider curvilinear coordinates $(\xi^1,\xi^2)$ on $\Sigma$
with vectors $\mbs{T}_\alpha, \alpha=1,2$ tangent to
the coordinate lines. Following \cite{Romero:2018iu},
the sought constraint can be expressed as
\begin{equation}
  \mbs{\theta}_*
  =
  \mbs{J}^{-1}
  \int_{\Sigma} \mbs{T}^\alpha \times \mbs{u}_{,\alpha}
  \,\mathrm{d} A ,
  \label{eq-constraint-theta}
\end{equation}
where the tensor $\mbs{J}$ is given by
\begin{equation}
  \mbs{J}
  :=
  \int_{\Sigma} (2\mbs{I} - \mbs{T}^{\alpha}\otimes \mbs{T}_\alpha)
  \,\mathrm{d} A
  \ ,
  \label{eq-inertia}
\end{equation}
the convention of sum over repeated indices is employed,
with $\alpha$ running from~1 to~2, and $\{\mbs{T}^\alpha\}_{\alpha=1}^2$
being the dual basis of the curvilinear coordinates.
See Appendix~\ref{app} for its derivation.

\subsection{Global problem statement}
In this joint problem, the equilibrium of the structure consisting
of the clamped beam, the deformable body and the connecting link
is obtained from the stationarity condition of a Lagrangian.
To define the latter, consider first the space of Lagrange multipliers 
\begin{equation}
  \mathcal{Q}
  \defined
  \mathbb{R}^3 \times\mathbb{R}^3
   \label{eq-space-lagmultipliers}
\end{equation}
with norm
\begin{equation}
  \normq{(\mbs{\lambda},\mbs{\mu})}
  \defined
  \left(
    \frac{1}{L^2}
    \norm{\mbs{\lambda}}^2_2
  +
  \norm{\mbs{\mu}}^2_2
  \right)^{1/2}
\ .
   \label{eq-norms-lagrange}
\end{equation}

Since the global problem involves two types of bodies, we start by defining 
one last product space $\mathcal{V} :=
\mathcal{U}\times\mathcal{W}\times\mathcal{R}$
with norm
\begin{equation}
  \normv{(\mbs{u}, \mbs{w}, \mbs{\theta})}
  \defined
  \left(
    \norm{\mbs{u}}^2_{\mathcal{U}}
    +
    \norm{\mbs{w}}^2_{\mathcal{W}}
    +
    L^2
    \norm{\mbs{\theta}}^2_{\mathcal{R}}
    \right)^{1/2},
   \label{eq-v-norm}
\end{equation}
for all $(\mbs{u},\mbs{w},\mbs{\theta})\in\mathcal{V}$. On this space,
we can define the bilinear form
$a(\cdot,\cdot):\mathcal{V}\times\mathcal{V}\to\mathbb{R}$
and the linear form $f:\mathcal{V}\to\mathbb{R}$ by
\begin{equation}
  \begin{aligned}
    a(\mbs{u},\mbs{w},\mbs{\theta}; \mbs{v}, \mbs{t}, \mbs{\eta})
    &\defined
    a_{\mathcal{B}}(\mbs{u},\mbs{v})
    +
    a_b(\mbs{w},\mbs{\theta}; \mbs{t}, \mbs{\eta})
    \ , \\
    f(\mbs{v},\mbs{t},\mbs{\eta})
    &\defined
    f_{\mathcal{B}}(\mbs{v})
    +
    f_b(\mbs{t},\mbs{\eta})\ .
  \end{aligned}
   \label{eq-bilinears}
\end{equation}

The joint equilibrium of the solid and beam will be obtained
as the saddle point of the Lagrangian
$L:\mathcal{V}\times\mathcal{Q}$
defined as
\begin{equation}
  \begin{aligned}
  L(\mbs{u},\mbs{w},\mbs{\theta},\mbs{\lambda},\mbs{\mu})
  \defined&
  \frac{1}{2} a(\mbs{u}, \mbs{w}, \mbs{\theta}; \mbs{u},\mbs{w}, \mbs{\theta})
  - f(\mbs{u}, \mbs{w}, \mbs{\theta})
  \\
  &+
  \langle
  \mbs{\lambda},
  \mbs{T}^{\alpha}\times  \mbs{u}_{,\alpha} -
  \mbs{J} \mbs{\theta}_*
  \rangle_\Sigma
  +
  \langle \mbs{\mu} , \mbs{u} - \mbs{w}_* \rangle_\Sigma .
  \end{aligned}
   \label{eq-lagrangian}
\end{equation}
where the notation $\langle\cdot,\cdot\rangle_\Sigma$ denotes the  $L_2$ product on
the surface $\Sigma$. The optimality conditions of the Lagrangian give the mixed
variational problem: find $(\mbs{u},\mbs{w},\mbs{\theta},\mbs{\lambda},\mbs{\mu}) \in
\mathcal{V}\times \mathcal{Q}$ such that
\begin{equation}
  \begin{aligned}
    a(\mbs{u},\mbs{w},\mbs{\theta}; \mbs{v}, \mbs{t}, \mbs{\beta})
    +
    b( \mbs{\lambda}, \mbs{\mu}; \mbs{v}, \mbs{t}, \mbs{\beta})
    &= f( \mbs{v}, \mbs{t}, \mbs{\beta}) ,
    \\
    b( \mbs{\gamma}, \mbs{\nu}; \mbs{u},\mbs{w},\mbs{\theta})
    &= 0, 
    \\
  \end{aligned}
   \label{eq-mixed}
\end{equation}
for all $(\mbs{v}, \mbs{t}, \mbs{\beta}, \mbs{\gamma}, \mbs{\nu})$
in $\mathcal{V}\times\mathcal{Q}$, with 
\begin{equation}
  \begin{aligned}
    b( \mbs{\gamma}, \mbs{\nu}; \mbs{u},\mbs{w},\mbs{\theta})
    &=
    \langle \mbs{\gamma},
    \mbs{T}^{\alpha}\times \mbs{u}_{,\alpha} - \mbs{J}\mbs{\theta}_*
    \rangle_\Sigma
    +
    \langle \mbs{\nu} , \mbs{u} - \mbs{w}_* \rangle_\Sigma\ .
  \end{aligned}
   \label{eq-forms}
\end{equation}
The solvability of problem~\eqref{eq-mixed} requires
the careful consideration of the properties of
both bilinear forms $a(\cdot,\cdot)$ and $b(\cdot,\cdot)$,
as well as the spaces on which they are defined.

\section{Analysis}
\label{sec-analysis}
Mixed variational problems such as the one described in Eqs.~\eqref{eq-mixed}
have been extensively studied in the literature \cite{Brezzi:1991tn,Roberts:1989vm}.
Their well-posedness pivots on two conditions: the ellipticity of
the bilinear form $a(\cdot,\cdot)$ on a certain set $\mathcal{K}\subset\mathcal{V}$
defined below, and the inf-sup condition of the bilinear form $b(\cdot,\cdot)$. 

Before stating the main result we note that, based on Eqs.~\eqref{solid-bounds}
and \eqref{eq-beam-bounds}, the global bilinear
form $a(\cdot,\cdot)$ verifies the following bounds
\begin{equation}
  \begin{aligned}
    |a(\mbs{u},\mbs{w},\mbs{\theta}; \mbs{v},\mbs{t},\mbs{\eta})|
    &\le 
    C \normv{(\mbs{u},\mbs{w},\mbs{\theta})}
    \normv{(\mbs{v},\mbs{t},\mbs{\eta})}\ , \\
    a(\mbs{u},\mbs{w},\mbs{\theta}; \mbs{u},\mbs{w},\mbs{\theta})
    +
    \norm{\mbs{u}}^2_{[L^2(\mathcal{B})]^3}
    &\ge
    \alpha
    \normv{(\mbs{u},\mbs{w},\mbs{\theta})}^2\ ,
    \\
  \end{aligned}
  \label{eq-compact}
\end{equation}
for some $C,\alpha>0$
and all $(\mbs{u},\mbs{w},\mbs{\theta}), (\mbs{v},\mbs{t},\mbs{\eta})\in\mathcal{V}$.
Likewise, and due to Eqs.~\eqref{eq-solid-fbound} and~\eqref{eq-beam-fbound}
the global linear form $f(\cdot)$ is continuous, that is,
\begin{equation}
   f( \mbs{v},\mbs{t},\mbs{\eta} )
    \le
    c \normv{ (\mbs{v},\mbs{t},\mbs{\eta})}\ .
  \label{eq-global-f-bound}
\end{equation}
for some $c>0$ and all $(\mbs{v},\mbs{t},\mbs{\eta})\in\mathcal{V}$.
We note, again, that the bilinear form $a(\cdot;\cdot)$ is not
coercive in $\mathcal{V}$, as a result of the lack of coercivity
of the bilinear form in the problem of the deformable solid.

The set $\mathcal{K}\subset\mathcal{V}$ consists of all
the functions where the bilinear form
$b(\cdot,\cdot)$ vanishes, i.e.,
\begin{equation}
  \mathcal{K}
  =
  \set{
    (\mbs{u},\mbs{w},\mbs{\theta})\in
    \mathcal{V}, 
    \
    b(\mbs{\gamma}, \mbs{\nu}; \mbs{u}, \mbs{w}, \mbs{\theta}) = 0
    \quad \textrm{for all}\ (\mbs{\gamma},\mbs{\nu})
    \in \mathcal{Q}
  }.
   \label{eq-kernel}
\end{equation}
From the definition of the bilinear form $b(\cdot;\cdot)$
it follows that the elements in $\mathcal{K}$ are 
ones that satisfy the constraints
\eqref{eq-constraint-w} and~\eqref{eq-constraint-theta}.

The well-posedness of the saddle point problem is the result of two theorems
that we state and prove next.

\begin{theorem}
  \label{thm-coercivity}
  The bilinear form $a(\cdot;\cdot)$ is $\mathcal{V}$-elliptic on $\mathcal{K}$.
\end{theorem}

\begin{proof}
Let the function $\norma{\cdot}:\mathcal{V}\to\mathbb{R}$ be defined as
\begin{equation}
  \norma{(\mbs{u},\mbs{w}, \mbs{\theta})}
  =
  a(\mbs{u},\mbs{w}, \mbs{\theta}; \mbs{u},\mbs{w}, \mbs{\theta})\ ,
  \label{eq-norma}
\end{equation}
for all $(\mbs{u},\mbs{w}, \mbs{\theta})\in\mathcal{V}$.  We prove first that this
function is positive definite on~$\mathcal{K}$. For $(\mbs{u},\mbs{w},\mbs{\theta})\in\mathcal{K}$,
$\norma{(\mbs{u},\mbs{w},\mbs{\theta})}=0$ if and only if
\begin{equation*}
  0
  =
  a_{\mathcal{B}}(\mbs{u},\mbs{u}) + a_b(\mbs{w},\mbs{\theta}; \mbs{w},\mbs{\theta})
  \ .
\end{equation*}
The bilinear forms $a_{\mathcal{B}}(\cdot,\cdot)$ and $a_b(\cdot,\cdot)$ are 
positive semidefinite and positive definite, respectively. Hence,
$(\mbs{w},\mbs{\theta})$ must be equal to $(\mbs{0},\mbs{0})$ and
$\mbs{u}$ must be an infinitesimal
rigid body motion. The only rigid body deformation in $\mathcal{K}$ is
\begin{equation*}
  \mbs{u} = \mbs{w}_* + \mbs{\theta}_*\times(\mbs{x}-\mbs{x}_G),
\end{equation*}
with $\mbs{x}_G$ the position of the center of area of $\Sigma$. But,
since $\mbs{w}\equiv\mbs{0}$ and $\mbs{\theta}\equiv\mbs{0}$, the function $\mbs{u}$
must also be identically zero.

To prove next that $\norma{(\mbs{u},\mbs{w},\mbs{\theta})} \ge \alpha 
\normv{(\mbs{u},\mbs{w},\mbs{\theta})}$ for some constant $\alpha>0$,
and any $(\mbs{u},\mbs{w},\mbs{\theta})\in\mathcal{K}$, suppose that
it is not true. Then there is a sequence $\{(\mbs{u}_i,\mbs{w}_i,\mbs{\theta}_i)\}
\in\mathcal{K}$ with
\begin{equation*}
  \normv{(\mbs{u}_i,\mbs{w}_i,\mbs{\theta}_i)} = 1,
  \qquad
  \hbox{and}
  \qquad
  \lim_{i\to\infty} \norma{(\mbs{u}_i,\mbs{w}_i,\mbs{\theta}_i)} = 0\ .
\end{equation*}
Since $1= \normv{(\mbs{u}_i,\mbs{w}_i,\mbs{\theta}_i)} \ge
\norm{\mbs{u}_i}_{[H^1(\mathcal{B})]^3}$, the sequence $\{\mbs{u}_i\}$ is
bounded in $[H^1(\mathcal{B})]^3$ and, by Rellich's theorem, there
is a subsequence $\{ \mbs{u}_{i_j}\}$ that 
converges in $[L^2(\mathcal{B})]^3$ to a function $\bar{\mbs{u}}$.
But, noting that
$ \lim_{j\to\infty} \norma{(\mbs{u}_{i_j},\mbs{w}_{i_j},\mbs{\theta}_{i_j})} = 0$,
this must be a Cauchy sequence in the norm
\begin{equation*}
  (\mbs{u},\mbs{w},\mbs{\theta})\mapsto
  \left(\norm{\mbs{u}}^2_{[L^2(\mathcal{B})]^3} + \norma{(\mbs{u},\mbs{w},\mbs{\theta})}^2\right)^{1/2}.
\end{equation*}
But this norm is equivalent to $\normv{\cdot}$ due to Korn's second inequality
and the ellipticity of $a_b(\cdot,\cdot)$. Hence, the sequence is Cauchy with respect
to $\normv{\cdot}$ and since $\mathcal{V}$ is a Hilbert space, 
it converges to $(\bar{\mbs{u}},\bar{\mbs{w}}, \bar{\mbs{\theta}})\in\mathcal{K}$.
The two norms being equivalent proves that
\begin{equation*}
  0 =
  \lim_{j\to\infty} \norma{(\mbs{u}_{i_j},\mbs{w}_{i_j},\mbs{\theta}_{i_j})}
  = 
  \norma{(\bar{\mbs{u}},\bar{\mbs{w}}, \bar{\mbs{\theta}})}.
\end{equation*}
Above we showed that $\norma{\cdot}$ is positive definite in $\mathcal{K}$,
hence $(\bar{\mbs{u}},\bar{\mbs{w}}, \bar{\mbs{\theta}})=(\mbs{0},\mbs{0},\mbs{0})$
but 
\begin{equation*}
  0
  =
  \norm{(\bar{\mbs{u}},\bar{\mbs{w}}, \bar{\mbs{\theta}})}
  =
  \lim_{j\to\infty} \norma{(\mbs{u}_{i_j},\mbs{w}_{i_j},\mbs{\theta}_{i_j})}
  = 1\ .
\end{equation*}
Since this is impossible, we conclude that there exists $\alpha>0$ such
that $\norma{(\mbs{u},\mbs{w},\mbs{\theta})} \ge \alpha 
\normv{(\mbs{u},\mbs{w},\mbs{\theta})}$.\qed
\end{proof}

The second condition required to guarantee the well-posedness of the mixed
problem is the \emph{inf-sup} condition on the bilinear form $b(\cdot,\cdot)$.

\begin{theorem}
  \label{thm-inf-sup}
There exists a constant $\beta>0$ such that
for all $(\mbs{\lambda},\mbs{\mu})\in\mathcal{Q}$,
\begin{equation}
  \sup_{(\mbs{u}, \mbs{w}, \mbs{\theta})\in\mathcal{V}} 
  \frac{b(\mbs{\lambda}, \mbs{\mu}; \mbs{u}, \mbs{w}, \mbs{\theta})}{
    \normv{(\mbs{u},\mbs{w},\mbs{\theta})}}
  \ge 
  \beta
  \normq{(\mbs{\lambda},\mbs{\mu})} .
   \label{eq-inf-sup}
\end{equation}
\end{theorem}

\begin{proof}
  To prove this bound, we choose $\mathcal{V}\ni(\mbs{u},\mbs{w},\mbs{\theta}) =
  (\mbs{\mu}+ \mbs{\lambda}\times(\mbs{x}-\mbs{x}_G), \mbs{0},\mbs{0})$.
  Thus,
\begin{equation}
  \begin{split}
    \sup_{(\mbs{u}, \mbs{w}, \mbs{\theta})\in\mathcal{V}} 
    \frac{b(\mbs{\lambda}, \mbs{\mu}; \mbs{u}, \mbs{w}, \mbs{\theta})}{
    \normv{(\mbs{u},\mbs{w},\mbs{\theta})}}
  &\ge 
   \frac{b(\mbs{\lambda}, \mbs{\mu}; \mbs{\mu}+
     \mbs{\lambda}\times(\mbs{x}-\mbs{x}_G), 
     \mbs{0}, \mbs{0})}{
    \normv{(\mbs{u},\mbs{w},\mbs{\theta})}}
  \\
  &=
  \frac{
    \langle 
    \mbs{\lambda} ,
    \mbs{T}^\alpha\times(\mbs{\lambda}\times \mbs{T}_\alpha)
    \rangle_{\Sigma}
    +
    \langle
    \mbs{\mu},
    \mbs{\mu}+
     \mbs{\lambda}\times(\mbs{x}-\mbs{x}_G)
     \rangle}{
     \normu{\mbs{\mu}+\mbs{\lambda}\times(\mbs{x}-\mbs{x}_G)}}
  \\
  &=
  \frac{
    \mbs{\lambda}\cdot \mbs{J}\mbs{\lambda}
    +
    |\mbs{\mu}|^2 |\Sigma|}
  {\left( |\mbs{\mu}|^2 |\mathcal{B}| + \mbs{\lambda}\cdot \mbs{M} \mbs{\lambda}\right)^{1/2}}
  \\
  &\ge
  \beta \normq{(\mbs{\lambda},\mbs{\mu})},
  \end{split}
   \label{eq-proof-infsup}
\end{equation}  
with the inertia
\begin{equation}
  \mbs{M}
  \defined
  \int_V
  \left(
    |\mbs{x}- \mbs{x}_G|^2 \mbs{I} 
    - (\mbs{x}-\mbs{x}_G)\otimes (\mbs{x}-\mbs{x}_G)
  \right)
  \,\mathrm{d} V
  \label{eq-inertia2}
\end{equation}
where we have employed the boundedness of $\mathcal{B}$ and~$\Sigma$.
\qed
\end{proof}

Theorems~\ref{thm-coercivity} and~\ref{thm-inf-sup} are
necessary and sufficient conditions for the well-posedness of
problem~\eqref{eq-mixed}.

\section{Summary}
\label{sec-summary}
We have presented the small strain form of a variational principle
that governs the collective equilibria of linked beams and deformable
solids. This is a remarkable principle in that it combines the
mechanical response of two types of bodies with very different
kinematic descriptions.

The variational principle rests on two compatibility conditions
that link, in the weakest possible way, the kinematics of beams
and solids on their common interface. While the compatibility
of translations is fairly straightforward, the compatibility of
beam rotations and displacements of the solid's surface is new
and based on a recent work of the author \cite{Romero:2018iu}.

The optimality conditions of this variational principle give
rise to a saddle point problem whose well-posedness is proven.
In addition to the mathematical consequences of such a result,
it evinces that it can be the basis of convergent numerical
discretizations for structural models combining beams and
deformable solids.

We close by noting that the well-posedness of the problem
does not rely on the elastic response of either the solid
or the beam. Rather, only some (weak) coercivity conditions
of the bilinear forms of the solid and the beam are required for the proof.
Hence, the result obtained can be, in principle, extended
to inelastic structures in which the same
stability estimates hold, even if just incrementally.

\appendix
\section{Derivation of the rotational constraint}
\label{app}
We derive next an intrinsic form of the constraint that
links the rotation vector at the tip of the beam, denoted as $\mbs{\theta}_*$,
with the \emph{average} rotation of the surface $\Sigma\subset\partial\mathcal{B}$.

Following \cite{Romero:2018iu}, we consider first the large strain
case. For that, we define the \emph{surface} deformation gradient.
Given a solid with reference configuration $\mathcal{B}_0$ and
a surface $\Sigma\subsetneq\partial\mathcal{B}_0$ with
curvilinear coordinates $(\xi^1,\xi^2)$,
and tangent vectors $\mbs{T}_\alpha,\alpha=1,2$, 
the surface deformation gradient is the map
\begin{equation}
  \mbs{f}
  \defined
  \pd{\varphi^i}{\xi^\alpha} \mbs{e}_i\otimes \mbs{T}^\alpha\ ,
  \label{eq-app-f}
\end{equation}
where $\mbs{\varphi}:\mathcal{B}_0\to\mathbb{R}^3$ is the deformation
of the solid, $\{\mbs{e}_i\}_{i=1}^3$ is a basis of $\mathbb{R}^3$,
and $\mbs{T}^\alpha,$ with $\alpha=1,2$ is the dual coordinate basis on the
reference surface~$\Sigma$.

Since the surface deformation gradient has a unique polar decomposition~\cite{RomArr09}
$\mbs{f}=\mbs{R}\mbs{U}$, with $\mbs{R}\in SO(3)$ and $\mbs{U}$ a rank-two symmetric tensor, 
the rotation $\mbs{\Lambda}_*\in SO(3)$ at the tip of the beam is equal to the
average rotation of the surface deformation gradient $\mbs{f}$ if
and only if
\begin{equation}
  \mbs{0} =
  \frac{1}{|\Sigma|}
  \int_{\Sigma} \mathrm{skew}[ \mbs{\Lambda}_*^T \mbs{f}] \,\mathrm{d} A\ .
  \label{eq-app-def}
\end{equation}
To study the form of this constraint in the small strain regime, we linearize the
integrand of Eq.~\eqref{eq-app-def}. Using $\epsilon $ as a small
parameter, we can introduce the expansion
\begin{equation}
  \mbs{\Lambda}_*^T \mbs{f}
  =
  ( \mbs{I} + \epsilon \hat{\mbs{\theta}}_*)^T
  ( \mbs{I}_{\Sigma} + \epsilon \nabla \mbs{u}) + \bigO(\epsilon^2)
  =
  \mbs{I}_{\Sigma} + \epsilon \nabla \mbs{u} - \epsilon \hat{\mbs{\theta}}_*
  \mbs{I}_{\Sigma} + \bigO(\epsilon^2),
  \label{eq-expansion}
\end{equation}
with
\begin{equation}
  \mbs{I}_\Sigma \defined \delta^\beta_\alpha \mbs{T}_\beta\otimes \mbs{T}^\alpha,
  \qquad
  \nabla \mbs{u} \defined \pd{u^i}{\xi^\alpha} \mbs{e}_i\otimes \mbs{T}^\alpha\ ,
\end{equation}
and $\delta^\beta_\alpha$ being the Kronecker's delta. The tensor $\mbs{I}_\Sigma$ is the identity
tensor of tangent vectors to $\Sigma$. Combining Eqs.~\eqref{eq-app-def}
and~\eqref{eq-expansion}, the linearized rotational constraint is
\begin{equation}
  \int_{\Sigma} \mathrm{skew}[\nabla \mbs{u}]
  \,\mathrm{d} A
  =
  \int_\Sigma
  \mathrm{skew}[ \hat{\mbs{\theta}}_* \mbs{I}_{\Sigma} ]
  \,\mathrm{d} A\ ,
\end{equation}
or equivalently,
\begin{equation}
  \int_\Sigma \mbs{T}^\alpha \times
  \left(
    \pd{u^i}{\xi^\alpha} \mbs{e}_i
  \right)
  \,\mathrm{d} A
  =
  \int_\Sigma \mbs{T}^\alpha\times (\mbs{\theta}_* \times \mbs{T}_\alpha)
  \,\mathrm{d} A\ .
  \label{eq-app-pr}
\end{equation}
By defining the section tensor $\mbs{J}$ as in Eq.~\eqref{eq-inertia},
Eq.~\eqref{eq-app-pr} can be rewritten as
\begin{equation}
  \mbs{\theta}_*
  =
  \mbs{J}^{-1}
  \int_\Sigma \mbs{T}^\alpha \times
  \left(
    \pd{u^i}{\xi^\alpha} \mbs{e}_i
  \right)
  \,\mathrm{d} A 
  \label{eq-app-final}
\end{equation}

This constraint links the rotation vector $\mbs{\theta}_*$ with
a certain average rotation of a general surface $\Sigma$.
When this surface is plane, as required to represent the cross
section of a beam, we can select the curvilinear coordinates
with a constant tangent basis so that Eq.~\eqref{eq-app-final}
can be written in the more compact way:
\begin{equation}
  \mbs{\theta}_*
  =
  \mbs{J}^{-1}
  \int_\Sigma \mbs{T}^\alpha \times
  \mbs{u}_{,\alpha}
  \,\mathrm{d} A .
  \label{eq-app-final2}
\end{equation}


\end{document}